\documentclass[11pt, twoside]{article}
\usepackage{amssymb, amsmath, amsthm}
\usepackage[left=27mm, right=27mm, bottom=30mm]{geometry}
\usepackage{inputenc}
\usepackage{fancyhdr}
\usepackage{tikz}
\usepackage{mathabx}
\usetikzlibrary{positioning}
\usepackage{titling}
\usepackage{hyperref}
\predate{}
\postdate{}
\usepackage{lipsum}
\newtheorem{theorem}{Theorem}
\newtheorem{lemma}[theorem]{Lemma}
\newtheorem{proposition}[theorem]{Proposition}

\usepackage{titling}
\usepackage{xcolor}
\usepackage{placeins}
\usepackage{tikz}
\usepackage{verbatim}
\usepackage{titling}
\usepackage[T1]{fontenc}
\usepackage{currvita}
\usepackage{bbm}
\usepackage{enumitem}
\newtheorem{corollary}[theorem]{Corollary}

\theoremstyle{definition}

\theoremstyle{remark}

\newcommand{\G}{\mathcal{G}}

\begin{document}
\newcommand{\Addresses}{
\bigskip
\footnotesize

\medskip

\noindent Jorge Cruz Chapital, \textsc{Department of Mathematics, University of Toronto, Toronto, M5S 2E4, Canada.}\par\noindent\nopagebreak\textit{Email address: }\texttt{cruz.chapital at utoronto.ca}

\medskip

\noindent Tom\'a\v s~Fl\'idr, \textsc{Peterhouse, University of Cambridge, CB2 1RD, UK.}\par\noindent\nopagebreak\textit{Email address: }\texttt{tf388@cam.ac.uk}

\medskip

\noindent Maria-Romina~Ivan, \textsc{Department of Pure Mathematics and Mathematical Statistics, Centre for Mathematical Sciences, Wilberforce Road, Cambridge, CB3 0WB, UK,} and\\\textsc{Department of Mathematics, Stanford University, 450 Jane Stanford Way, CA 94304, USA.}\par\noindent\nopagebreak\textit{Email addresses: }\texttt{mri25@dpmms.cam.ac.uk, m.r.ivan@stanford.edu}}

\pagestyle{fancy}
\fancyhf{}
\fancyhead [LE, RO] {\thepage}
\fancyhead [CE] {JORGE CRUZ CHAPITAL, TOM\'A\v S FL\'IDR AND MARIA-ROMINA IVAN}
\fancyhead [CO] {ALL ORDINALS ARE COP-ROBBER ORDINALS}
\renewcommand{\headrulewidth}{0pt}
\renewcommand{\l}{\rule{6em}{1pt}\ }
\title{\Large{\textbf{ALL ORDINALS ARE COP-ROBBER ORDINALS}}}
\author{JORGE CRUZ CHAPITAL, TOM\'A\v S FL\'IDR AND MARIA-ROMINA IVAN}
\date{ }
\maketitle
\begin{abstract}
The game of cops and robbers, played on a fixed graph $G$, is a two-player game, where the cop and the robber (the players) take turns in moving to adjacent vertices. The game finishes if the cop lands on the robber's vertex. In that case we say that the cop wins. If the cop can always win, regardless of the starting positions, we say that $G$ is a cop-win graph.

For a finite cop-win graph $G$ we can ask for the minimum number $n$ such that, regardless of the starting positions, the game will end in at most $n$ steps. This number is called the maximum capture time of $G$. By looking at finite paths, we see that any non-negative integer is the maximum capture time for a cop-win graph. 

What about infinite cop-win graphs? In this case, the notion of capture time is nicely generalised if one works with ordinals, and so the question becomes which ordinals can be the maximum capture time of a cop-win graph? These ordinals are called CR (Cop-Robber)-ordinals. In this paper we fully settle this by showing that all ordinals are CR-ordinals, answering a question of Bonato, Gordinowicz and Hahn.
\end{abstract}
\section{Introduction}
Let $G$ be a fixed simple graph without multiple edges. Two players, the cop and the robber, each pick a starting vertex, with the cop picking first. Then they move alternately, with the cop moving first: at each turn the player moves to an adjacent vertex or does not move. The game is won by the cop if he lands on the robber. We say that $G$ is \textit{cop-win} if the cop has a winning strategy.

A central notion for this problem is that of a `dominated' vertex. For a vertex $v\in V(G)$ we denote by $N[v]$ its \textit{closed neighborhood}: its neighbors, together with $v$. We say that a vertex $y$ \textit{dominates} a vertex $x$ if $N[x]\subseteq N[y]$. The finite cop-win graphs were characterised by Nowakowski and Winkler \cite{NW}. They showed that a finite graph $G$ is cop-win if and only if it is constructible, meaning that it can be built up from the one-point graph by repeatedly adding dominated vertices. There are multiple variations of this game for finite graphs, such as allowing more than one cop to play. We direct the interested reader to the book of Bonato and Nowakowski \cite{BN} for general background and a wealth of other results in the finite case.

We now turn to infinite graphs, where the game of cops and robbers has the exact same rules as before, but the above notions behave very differently. For example, \textit{a ray} -- infinite path with one end -- is constructible, but it is trivially not a cop-win graph. Conversely, there exist cop-win graphs that are not constructible -- for this, and some other phenomena, see Ivan, Leader and Walters \cite{ILW}. We mention that currently, unlike the finite graphs setup, there is no structural characterization for infinite cop-win graphs.

From now on, all the graphs will be assumed to be cop-win. For a finite graph, we can define \textit{the capture time} to be the length of the game, assuming both players play optimally. That means that the cop chooses her starting position as favorably as possible, and so does the robber. After that, every cop's next move will be to minimise the time she needs to capture the robber, and every robber's next move will be to escape for as long as possible. For example, if $G=P_{2n+1}$, the capture time is $n$. How can we generalise this notion for infinite graphs?

Let $G$ be a finite cop-win graph. We denote by $\eta(u,v)$ the time it will take a cop at $v$ to capture a robber at $u$, with the robber moving first. Let $\eta(v)=\sup_{u\in V(G)}\eta(u,v)$, or in words, how long the game lasts if the cop starts at $v$. Therefore, the capture time, denoted by $\eta(G)$ is $\min_{v\in V(G)}\eta(v)$. However, the quantity we will be interested in is \textit{the maximum capture time}, denoted by $\rho(G)$. This is equal to $\sup_{u,v\in V(G)}\eta(u,v)$, or the maximum length the game can possibly last.

In the realm of infinite graphs, all these notions generalise nicely if one works with ordinals. We make things precise below.

Let $G$ be a graph of cardinality $\aleph_{\beta}$, where $\beta\geq 0$ is an ordinal. Set $\omega(G)=\omega_{\beta+1}$, and define the relations $\{\leq\alpha\}_{\alpha<\omega(G)}$ on $V(G)$ as follows. If $u = v$, then $u\leq_0 v$. Also, $u\leq_{\alpha} v$ if for all $x\in N[u]$, there exists $y\in N[v]$ such that $x\leq_{\gamma} y$ for some $\gamma<\alpha$.

Now, for any two vertices $u, v$, we define $\eta(u, v)=\alpha $, where $\alpha$ is the minimum ordinal for which $u\leq_\alpha v$ holds. We observe that if $\eta(u, v)$ is finite, then it is precisely equal to the time it takes the cop at $v$ to catch the robber at $u$, assuming both play optimally, and the robber moves first. Define $\eta(v)$ to be the minimum ordinal $\alpha$ such that $u\leq_{\alpha} v$ holds for all $u\in V(G)$. Finally, define $\eta(G) = \min_{v\in V(G)}\eta(v)$. One can check that if $\eta(G)$ is finite, then it is precisely the capture time of the cop-win graph $G$. Let also $\rho(G) =\sup_{v\in V(G)}\eta(v)$. In the finite case, $\rho(G)$ is the maximum capture time over all initial positions of the cop.

One of the main questions is what `times', hence ordinals, can the maximum capture time, $\rho(G)$ be? Such ordinals are known as \textit{CR-ordinals}. By considering finite paths, every finite ordinal is a CR-ordinal. Bonato, Gordinowicz and Hahn \cite{BGH} showed that all ordinals in $\{\omega\cdot i + (i + j): i, j <\omega\}\cup\{\alpha +\omega: \alpha\text{ is a limit ordinal}\}$ are CR-ordinals, and conjectured that in fact these are all of them.

This conjecture was recently disproved by Fl\'{i}dr and Ivan \cite{FI} who showed that $\omega$ is in fact a CR-Ordinal. They constructed a cop-win graph such that the maximum capture time is $\omega$, but $\eta(u,v)\neq\omega$ for all $u$ and $v$.

Inspired by their construction, we show that, actually, all ordinals are CR-ordinals.

The plan of the paper is as follows. In Section 2 we deal with the case of infinite limit ordinals. The construction is inspired by \cite{FI}, where the vertex set is points with 2 coordinates less than the given limit ordinal, and the edges always go `up and to the left', and `down and to the right'. Additionally, the vertices on the $x$-axis, $y$-axis and the diagonal, each form a complete graph. In Section 3, by attaching a finite path to the origin of the graphs constructed in Section 2, we also deal with the case of an infinite successor ordinal.

\section{Maximum capture time as a limit ordinal}
Let $\gamma$ be an infinite limit ordinal, and $\mathcal{G}_\gamma$ be the graph with vertex set $V_\gamma=\gamma\times \gamma$ and edge set $E_\gamma$, where $\{(\alpha_0,\beta_0),(\alpha_1,\beta_1)\}\in E_{\gamma}$ if and only if $\alpha_0=\alpha_1=0$, or $\beta_0 = \beta_1 = 0$, or $\alpha_0=\beta_0$ and $\alpha_1=\beta_1$, or $\alpha_0<\alpha_1$ and $\beta_0>\beta_1$, or $\alpha_0>\alpha_1$ and $\beta_0<\beta_1$. In other words, any two vertices on the $x$-axis or on the $y$-axis are connected, any two vertices on the diagonal are connected, and additionally any vertex is connected to all vertices below it and to its right, as well as to all vertices above it and to its left.

\begin{proposition}\label{cop-win}
The graph $\mathcal G_{\gamma}$ is cop-win.
\end{proposition}
\begin{proof}
Let the cop be at some vertex $v_0$. From there, in her first move, she can go to a vertex on the diagonal, which we call $v_1$. Since any two vertices on the diagonal are connected, the robber at his original position $u_0$ has to move to a vertex $u_1=(\alpha_1,\beta_1)$ such that $\alpha_1\neq\beta_1$. Without loss of generality, we may assume that $\alpha_1>\beta_1$. Next, the cop moves to a vertex $v_2=(\xi_2,0)$ with $\xi_2>\alpha_1$. In order to evade the cop, the robber must move to a vertex $u_2=(\alpha_2,\beta_2)$ outside of $N[v_2]$. By construction, if $u_2\in N[u_1] \setminus N[v_2]$, then $\alpha_2\geq\xi_2$ and $\beta_2<\beta_1$ (which in particular means that $\alpha_2>\beta_2$). Continuing with this strategy, with the cop always staying on the $x$-axis, she can force the robber to always move down and to the right, decreasing his $y$-coordinate, and thus creating a strictly decreasing sequence of ordinals. Since there does not exist an infinite strictly decreasing sequence of ordinals, there exists a finite moment when the robber must reach the $x$-axis, where he will be caught in the next move.
\end{proof}

We note that although the strategy described above is a winning strategy for the cop, it is not necessarily optimal. Since we want to precisely calculate $\rho(\mathcal{G}_\gamma)$, we must find optimal `strategies' for any starting position. That is, we need both lower and upper bounds for all $\eta(u,v)$ where $u$ and $v$ are vertices of $\mathcal G_{\gamma}$. 

We first focus on the upper bounds, with all cases covered in the following three lemmas.

\begin{lemma}\label{lemmaxaxis} Let $u=(\alpha,\beta)$ and $v=(\xi,0)$ be distinct vertices of $\mathcal G_{\gamma}$ such that $0\leq\beta< \alpha<\xi$. If $\beta=0$, then $u\leq_1 v$, and if $0<\beta$, then $u\leq_\beta v$.  In particular, $\eta(u,v)\leq \beta+1$.\end{lemma}
\begin{proof} We prove the statement by induction on $\beta$. If $\beta=0$ or $\beta=1$, the result follows from the fact that $u$ is dominated by $v$. Now, assume that the lemma is true for all $1\leq\beta'<\beta.$ We have to show that for every  $x\in N[u]$, there exists $y\in N[v]$ such that $x\leq_\delta y$ for some $\delta<\beta$. Let $x=(\alpha_0,\beta_0)$. Indeed, if $x\in N[v]$ we take $y=x$ and $\delta=0$. If $x\notin N[v]$, then by construction $\alpha_0\geq\xi$ and $\beta_0>0$. Since $(\alpha_0,\beta_0)$ and $(\alpha,\beta)$ are adjacent, and $\alpha_0\geq\xi>\beta>\alpha$, we must have that $0<\beta_0<\beta$.  In this case, by the induction hypothesis, setting $y=(\alpha_0+1,0)$, we have $y\in N[v]$ and $x\leq_{\beta_0}y$, which finishes the proof of the claim.
\end{proof}

By the duality of the construction, we also have the following.

\begin{lemma}\label{lemmayaxis} Let $u=(\alpha,\beta)$ and $v=(0,\xi)$ be two distinct vertices of $\mathcal G_{\gamma}$ such that $0\leq \alpha< \beta<\xi$. If $\alpha=0$, then $u\leq_1 v$. If $0<\alpha$, then $u\leq_\alpha v$. In particular, $\eta(u,v)\leq \alpha+1$.
\end{lemma}
We are left to analyse the game time given the starting positions on the diagonal.
\begin{lemma}\label{lemmadiagonal}Let $u=(\alpha,\alpha)$ and $v=(\xi,\xi)\in V_\gamma$ be two vertices on the diagonal. If $\alpha\leq1$, then $u\leq_2 v$. If $\alpha>1$, then $u\leq_\alpha v $. In particular, $\eta(u,v)\leq \alpha+2$.\end{lemma}
\begin{proof}If $\alpha\leq 1$, the result is trivial since, if the robber is at $(\alpha,\alpha)$ and it is his turn to move, he can either move on the diagonal, or on one of the axes. In either case, he is captured in at most 2 turns. Therefore we may assume that $\alpha>1$. We have to show that for every $x=(\alpha_0,\beta_0)\in N[u]$, there exists $y\in N[v]$ such that $x\leq_\delta y$ for some $\delta<\alpha$. If $x\in N[v]$, as discussed above, there is nothing to show. Thus, we may assume that $x\notin N[v]$. This, together with the fact that $x$ is adjacent to $u$ means that either $\beta_0<\alpha<\alpha_0$, or $\alpha_0<\alpha<\beta_0$. If we are in the first case, let $y=(\xi_0,0)\in N[v]$ be such that $\xi_0>\alpha_0$, and if we are in the second case, let $y=(0,\xi_0)\in N[v]$ be such that $\xi_0>\beta_0$. Let $\delta=1$ if $\min(\alpha_0,\beta_0)\leq 1$ and $\delta=\min(\alpha_0,\beta_0)$ otherwise. Note that $\delta<\alpha$ by construction. By Lemma~\ref{lemmaxaxis} and Lemma~\ref{lemmayaxis}, we have that $x\leq_\delta y$, which finishes the proof.
\end{proof}
We therefore have the following corollary, which is half way towards showing that the maximum capture time of $\mathcal G_{\gamma}$ is $\gamma$.
\begin{corollary} For any two vertices $u$ and $v$ of $\mathcal G_{\gamma}$, we have that $\eta(u,v)\leq \gamma$. In particular, $\rho(\mathcal{G}_\gamma)\leq \gamma$.
\label{upperbound}
\end{corollary}
\begin{proof}Let $x=(\alpha,\beta)\in N[u]$. We must show that there exists $y\in N[v]$ such that $x\leq_\delta y$ for some $\delta<\gamma$. If $\alpha=\beta$, let $y\in N[v]$ be on the diagonal. By Lemma \ref{lemmadiagonal}, we have that $x\leq_{\alpha+2} y$, which is enough as $\alpha+2<\gamma$. If $\alpha\neq\beta$, let $\delta=\min(\alpha,\beta)+1$, which is strictly less than $\gamma$. By construction, from $v$ we can always reach a far away point $y$ on either one of the axes such that the hypothesis of either Lemma \ref{lemmaxaxis} or Lemma \ref{lemmayaxis} is fulfilled for $x$ and $y$. Therefore we have that $x\leq_\delta y$, which finishes the proof.
\end{proof}

We now move on to lower bounds for the quantities $\eta(u,v)$. We will in fact show that for any two distinct vertices $u=(\alpha,\beta)$ and $v$, $\eta(u,v)$, is at least $\min(\alpha,\beta)$. One would expect the proof to be by induction on $\alpha$ and/or $\beta$, but this proves problematic as the neighborhood of $v$ always contains vertices whose first coordinate is bigger than $\alpha$, as well as vertices whose second coordinate is bigger than $\beta.$ We get around this issue by endowing the vertex set $V_{\gamma}$ with a well-founded order $\preceq$, for which induction will be much easier in this case.

Let $u$ be a vertex of $\mathcal G_{\gamma}$. We define $f_u:N[u]\times V_{\gamma}\longrightarrow  V_{\gamma}$ to be a function such that for all $x\in N[u]$ and  all $v\in V_{\gamma}$, we have $f_u(x,v)\in N[v]$, and  $\eta(x,f_u(x,v))<\eta(u,v)$ if $u\neq v$. If $u=v$, we define $f_u(x,v)=v$.

In  other words, the function $f_u$ is choosing an optimal move for the cop, after the robber at $u$ has made his move. Note that the existence of the functions $f_u$ for all $u\in V_{\gamma}$ is guaranteed by the fact that $G_{\gamma}$ is cop-win. 

We now define the well-founded partial order  $\preceq$ on the set of ordered pairs of vertices $V_{\gamma}\times V_{\gamma}$ as follows. We say that $(u,v)\preceq(w,z)$ if and only if  either $(u,v)=(w,z)$, or there exists $n\geq2$ and $(u,v)=(u_1,v_1),\dots,(u_n,v_n)=(w,z)$ such that $f_{u_{i+1}}(u_i,v_{i+1})=v_i$ for all $i<n$. In other words, a set of initial positions $(w,z)$ for the players is `bigger' than another set $(x,y)$, if the game can be played optimally from $(w,z)$ and reach $(x,y)$.

It is clear that $\preceq$ is a partial order. Furthermore, we observe that if $(u,v)\prec (w,z)$  we have that that $\eta(u_i,v_i)<\eta(u_{i+1},v_{i+1})$ for all $i<n$, which gives $\eta(u,v)<\eta(z,w)$. Therefore $(V_{\gamma}\times V_{\gamma},\preceq)$ is a well-founded order.

For any $(u,v)\in V_{\gamma}\times V_{\gamma}$, looking at a set of initial positions just `below' $(u,v)$ (with respect to $\preceq$), it is straightforward from the definitions that
\begin{align*}
\eta(u,v)&=\sup\{\eta(w,z)+1:(w,z)\precneq(u,v)\}\\
&=\sup\{\eta(w,z)+1:w\in N[u]\text{ and }z=f_u(w,v)\}.
\end{align*}

\begin{lemma}\label{etalowerboundlemma}Let $u=(\alpha,\beta)$ and $v=(\xi,\delta)$ be two distinct vertices of $\mathcal G_{\gamma}$. Then $\eta(u,v)\geq \min(\alpha,\beta)$.
\end{lemma}
\begin{proof}We prove this by induction over the well-founded order $(V_{\gamma}\times V_{\gamma},\prec)$. If $\eta(u,v)=1$ (the base case), then $u$ is dominated by $v$. It is easy to see that if both coordinates of $u$ are greater than 1, then $u$ is not dominated by any other vertex. We therefore have $\eta(u,v)=1\geq\min(\alpha,\beta)$, as claimed. Let now $u=(\alpha,\beta)$ and $v=(\xi,\delta)$ be two distinct vertices such that $\eta(u,v)>1$, and assume that the lemma is true for any two  distinct vertices $u'$ and $v'$ such that $(u',v')\precneq (u,v)$. We have to show that $\eta(u,v)\geq \min(\alpha,\beta)$. If $\alpha\leq1 $ or $\beta\leq 1$, we are done. Thus, we may assume that $\alpha,\beta\geq 2$. We divide the rest of the proof into cases.

Suppose first that either $\xi>\alpha$ or $\delta<\beta$. In this case we have that for every $\mu<\beta$ we can find $\mu \leq \beta_\mu<\beta$ and $\alpha_\mu>\alpha$  such that $u_\mu=(\alpha_\mu,\beta_\mu)\in N[u]\setminus N[v]$. Indeed, if $\xi>\alpha$ and $\delta\geq \beta$, we can take $\alpha_\mu=\xi$ and $\beta_\mu=\mu$. It is easy to see that in this case $v=(\xi,\delta)$ is not adjacent to $(\xi,\mu)\in N[u]$. If $\delta<\beta$, we can take $\alpha_\mu=\max(\alpha,\xi,\mu,\delta)+1$ and $\beta_\mu=\max(\mu,\delta)$, unless $\delta=\mu=0$, in which case we take $\beta_\mu=1$. 

Now,  let us take $f_u(u_\mu,v)=v_\mu$. According to the definition of $f_u$, $v_\mu$ lies inside $N[v]$. Thus, it is distinct from $u_{\mu}$. Moreover, we have that $(u_\mu,v_\mu)\precneq (u,v)$. By the inductive hypothesis we have that $\eta(u_\mu,v_\mu)\geq\min(\alpha_\mu,\beta_\mu)\geq \min(\alpha,
\mu)$. In this way,  $$\eta(u,v)\geq \sup\{\eta(u_\mu,v_\mu)+1:\mu<\beta\}\geq\sup\{\min(\alpha,\mu)+1:\mu<\beta\}\geq\min(\alpha,\beta).$$

Finally, assume that $\alpha\geq \xi$ and $\beta\leq \delta$. In this case, for every $\mu<\alpha$ we can find two ordinals $\alpha_\mu$ and $\beta_\mu$ (in $\gamma$) such that $\mu\leq \alpha_\mu<\alpha$, $\beta_\mu>\beta$, and $u_\mu=(\alpha_\mu,\beta_\mu)\in N[u]\setminus N[v]$. 

Indeed, if $\xi=\alpha$, then $\beta<\delta$ (since $u\neq v$), and so we may take $\alpha_\mu=\mu$ and $\beta_\mu=\delta$. If $\xi<\alpha$, we take $\alpha_\mu=\max(\mu,\xi)$ and $\beta_\mu=\max(\delta,\mu,\xi)+1$, unless $\xi=\mu=0$, in which case we take $\alpha_\mu=1$.

Therefore, by considering $v_\mu=f_u(u_\mu,v)$ we have that $v_\mu\neq u_\mu$ and $(u_\mu,v_\mu)\precneq (u,v)$. Thus, $\eta(u_\mu,v_\mu)\geq \min(\mu,\beta)$. Similarly as above we have
$\eta(u,v)\geq \sup\{\eta(u_\mu,v_\mu)+1:\mu<\beta\}\geq\sup\{\min(\mu,\beta)+1:\mu<\alpha\}\geq\min(\alpha,\beta).$
\end{proof}

\begin{lemma}Let $u=(\alpha,\alpha)$ be a vertex of $\mathcal G_{\gamma}$ on the diagonal. Then $\eta(u,v)=\gamma$ for all $v\in V_{\gamma}$ that do not lie on the diagonal.\label{lowerbound}
\end{lemma}
\begin{proof} Since by Corollary~\ref{upperbound} we have that $\eta(u,v)\leq\gamma$, it is enough to show that $\eta(u,v)\geq \gamma$. Indeed, since $v$ does not lie on the diagonal, for every $\mu<\gamma$, there exists $\mu<\alpha_\mu<\gamma$ such that  $u_\mu=(\alpha_\mu,\alpha_\mu)\notin N[v]$. Consider $v_\mu=f_u(u_\mu,v)$, which is distinct from $u_\mu$ as it lies in $N[v]$. By Lemma~\ref{etalowerboundlemma} we have that $\eta(u_\mu,v_\mu)\geq \alpha_\mu>\mu$. As above, we have that $$\eta(u,v)\geq \sup\{\eta(u_\mu,v_\mu)+1:\mu<\gamma\}\geq \gamma.$$
\end{proof}  
Finally, Corollary~\ref{upperbound} and Lemma~\ref{lowerbound} give the main result of the section.
\begin{theorem} The maximum capture time of $\mathcal G_{\gamma}$ is $\gamma$. In other words, $\rho(\mathcal{G}_\gamma)=\gamma$. \label{mainresultlimit}
\end{theorem}
\section{Maximum capture time as a successor ordinal}

Let $\gamma'$ be an infinite successor ordinal. In this section, building on the work done in Section 2, we construct a cop-win graph whose maximum capture time is $\gamma'$. 

First, since $\gamma'$ is a successor, there exists a positive integer $n$ and an infinite limit ordinal $\gamma$ such that $\gamma'=\gamma+n$. We define $\mathcal{G}_{\gamma+n}$ to be the graph with vertex set $V_{\gamma+n}=(\gamma\times \gamma)\cup(\{-(n+1),\dots,-1\}\times \{0\})$ and edge set consisting of all edges of $G_{\gamma}$, together with all the pairs $\{(-i,0),(-i+1,0)\}$ for all $i\in[n+1]$.

It is easy to see that $\mathcal{G}_{\gamma+n}$ is cop-win. Indeed, let the game be played on $\mathcal G_{\gamma+n}$, but the cop plays inside $\G_{\gamma}$, pretending that the robber is at $(0,0)$ whenever he is at $(-i,0)$ for some $i\in[n+1]$. Following the winning strategy described in Proposition~\ref{cop-win}, the cop captures the robber on the $x$-axis, or on the $y$-axis. The only way the game is not over is if the robber is at some $(-i,0)$ for some $i\in[n+1]$, in which case the cop follows him down the path $\{(-i,0):0\leq i\leq n+1\}$, eventually catching him.

Moreover, one can easily check that Lemma~\ref{lemmaxaxis} and Lemma~\ref{lemmayaxis} are still true, of course, interpreting now $\eta(u,v)$ inside of $\mathcal{G}_{\gamma+n}$, rather that $\mathcal G_\gamma$. Furthermore, Lemma~\ref{lemmadiagonal} transforms into the following.

\begin{lemma}\label{lemmadiagonal2} Let $u=(\alpha,\alpha)$ and $v=(\xi,\xi)\in V_{\gamma+n}$ be two vertices on the diagonal. Then $u\leq_\alpha v $ if $\alpha>1$, and $u\leq_{n+2} v$ if $\alpha\leq1$. In particular, $\eta(u,v)\leq \alpha+n+2$.
\end{lemma}
\begin{proof}If $u=v$ there is nothing to show, so we may assume that $\alpha\neq\xi$. If $\alpha>1$, then the proof is the same as in the proof of Lemma~\ref{lemmadiagonal}. If $\alpha=0$ and the robber does not move to $(-1,0)$, he is caught in at most two turns. If the robber moves to $(-1,0)$, the cop moves to $(0,0)$, and then the game will end in at most $n+1$ turns. Therefore, in this case we get $u\leq_{n+2}v$. If $\alpha=1$, the robber can be caught in at most 2 turns, hence $u\leq_2 v$ in this case, which finishes the proof.
\end{proof}
These three lemmas put together give us the following result.
\begin{lemma}\label{firsthalf} Let $u\in V_{\gamma+n}$ and $v\in\gamma\times\gamma$. Then $\eta(u,v)\leq \gamma+1$. Moreover, $\eta(u,(0,0))<\gamma$.
\end{lemma}
\begin{proof}
We first look at $\eta(u, (0,0))$, and observe that if $u=(-i,0)$ for some $0\leq i\leq n+1$, then $\eta(u, (0,0))$ is finite, and so less than $\gamma$. We may therefore assume that $u=(\alpha,\beta)\in(\gamma\times\gamma)\setminus\{(0,0\})$. Let $x\in N[u]$, which in this case implies that $x\in\gamma\times\gamma$. If $x$ is on the diagonal, then it is also in $N[(0,0)]$. Otherwise, let $x=(\alpha',\beta')$ for some $\alpha'\neq\beta'$, and $y\in N[(0,0)]$ on one of the axes (the $x$-axis if $\beta'<\alpha'$, and the $y$-axis otherwise) with the non-zero coordinate bigger than $\max(\alpha',\beta')$. We then get that $\eta(x,y)\leq\min(\alpha',\beta')+1\leq\max(\alpha,\beta)+1$, which consequently implies that $\eta(u, (0,0))\leq\max(\alpha,\beta)+2<\gamma$.

We now move on to $\eta(u,v)$ where $u$ is any vertex of $\mathcal G_{\gamma+n}$ and $v\in\gamma\times\gamma$. If $u$ is not of the form $(-i,0)$ for some $0\leq i\leq n+1$, then $\eta(u,v)\leq\gamma$ in the same way as in the proof of Corollary~\ref{upperbound}. Assume now that $u=(-i,0)$ for some $0\leq i\leq n+1$. We first observe that if $y$ is on the diagonal, then $\eta ((-i,0),y)$ is finite for all $1\leq i\leq n+1$. This, together with the previous observations give that $\eta((0,0), v)\leq\gamma$ for all $v\in\gamma\times\gamma$. Next, for $\eta((-i,0),v)$ where $i\geq 2$, since all neighbors of $(-i,0)$ are of the form $(-j,0)$ for some $1\leq j\leq n+1$, and from $v$ we can reach some diagonal point, we get that $\eta((-i,0),v)$ is finite. Finally, for $\eta((-1,0),v)$, in light of the above observation, the only neighbor of $(-1,0)$ to consider is $(0,0)$, for which $\eta((0,0),y)\leq\gamma$ for all $y\in\gamma\times\gamma$. This gives $\eta((-1,0)),v)\leq\gamma+1$, which finishes the proof of the lemma.
\end{proof}

In order to bound from above the maximum capture time of $\mathcal{G}_{\gamma+n}$, we now only need to bound from above $\eta(u,v)$, where $u\in V_{\gamma+n}$ and $v=(-i,0)$ for some  $1\leq i\leq n+1$. We do this in the next lemma.

\begin{lemma}\label{secondhalf}Let $u=(\alpha,\beta)$ and $v=(-i,0)\in V_{\gamma+n}$ be two vertices such that $1\leq i\leq n+1$. Then $\eta(u,v)\leq \gamma+(i-1)$.\end{lemma}
\begin{proof} We prove this by induction on $i$. If $i=1$ the result is trivial as $(0,0)\in N[v]$ and $\eta(x,(0,0))<\gamma$ for all $x\in V_{\gamma+n}$, in particular for all $x$ in $N[u]$. Assume now that $i>1$ and that the result is true for $i-1$. Since $y=(-(i-1),0)\in N[v]$ and, by the induction hypothesis we have $\eta(x,y)\leq\gamma+(i-2)$ for all $x\in V_{\gamma+n}$, we get that $\eta(x,y)\leq\gamma+(i-2)$ for all $x\in N[u]$. This implies that $\eta(u,v)\leq\gamma+(i-1)$, which finishes the proof of the lemma.
\end{proof}
Putting together Lemma~\ref{firsthalf} and Lemma~\ref{secondhalf}, we get the desired upper bound on $\rho(\mathcal G_{\gamma+n})$.
\begin{corollary} For any two vertices $u$ and $v$ of $\mathcal G_{\gamma+n}$, we have that $\eta(u,v)\leq \gamma+n$. In particular, $\rho(\mathcal{G}_\gamma)\leq \gamma+n$.\label{upperbound2}
\end{corollary}

Given this, it is now enough to show that there exist two vertices $u,v\in V_{\gamma+n}$ such that $\eta(u,v)=\gamma+n$. We achieve this via a straightforward modification of Lemma~\ref{etalowerboundlemma}, whose proof we omit since the argument is completely analogous.
\begin{lemma}\label{etalowerboundlemma2}Let $u=(\alpha,\beta)\in \gamma\times \gamma$  and $v\in V_{\gamma+n}$ be two distinct vertices. Then $\eta(u,v)\geq \min(\alpha,\beta)$.
\end{lemma}

\begin{lemma}Let $u=(-i,0)$ and $v=(-j,0)\in V_{\gamma+n}$ be two vertices such that $0\leq i<j\leq n+1$.  Then $\eta(u,v)\geq \gamma+i$. In particular, $\eta((-n,0),(-(n+1),0))=\gamma+n$.\label{loweround2}
\end{lemma}
\begin{proof} We prove this by induction on $i$. If $i=0$, we need to prove that $\eta(u,v)\geq \gamma$. Indeed, given $\mu<\gamma$ let $x_\mu=(\mu,\mu)$ which is in $N[u]$. By Lemma~\ref{etalowerboundlemma2} we get that $\eta(x_\mu,y)\geq \mu$ for any $y\neq x_\mu$, and in particular for any $y\in N[v]$. As $\mu$ was taken arbitrary, we indeed have that $\eta(u,v)\geq \gamma$. Now assume that $i>0$ and that the lemma holds for $x=(-(i-1),0)$ and any vertex of the form $(-j',0)$ with $n+1\geq j'>i-1$. We notice that $x=(-(i-1),0)\in N[u]$ and any $y\in N[v]$ is of the form $(-j',0)$ for some $n+1\geq j'>i-1$. By the inductive hypothesis we have that $\eta(x,y)\geq \gamma+(i-1)$ for any  $y\in N[v]$, which consequently implies that $\eta(u,v)\geq \gamma+i$, finishing the proof of the lemma.
\end{proof}
Corollary~\ref{upperbound2} and Lemma~\ref{loweround2} give us the following.
\begin{theorem} The maximum capture time of $\mathcal G_{\gamma+n}$ is $\gamma+n$. In other words, $\rho(\mathcal{G}_{\gamma+n})=\gamma+n$.\label{mainresultsuccessor}   
\end{theorem}
Theorem~\ref{mainresultlimit} and Theorem~\ref{mainresultsuccessor}, together with the observation that any finite ordinal can represent the maximum capture time of a cop-win graph, give the main result of the paper.
\begin{theorem} Let $\alpha$ be an ordinal. Then there exists a cop-win graph $G$ such that $\rho(G)=\alpha$, i.e. the maximum capture time of $G$ is $\alpha$.\label{main}
\end{theorem}
\section{Concluding remarks}

We remark that if $\gamma=\omega$, the the graph $\mathcal G_\omega$ we have constructed in Section 2 answers positively Question 7 in \cite{FI}. More precisely, not only the maximum capture time of $\mathcal G_\omega$ is $\omega$, but there exist two vertices $u$ and $v$ such that $\eta(u,v)=\omega$.

We recall that if $\eta(x,y)$ is finite, then it is equal to the number of moves a cop at $y$ needs to catch a robber at $x$, with the robber moving first. Since all the ordinals below $\omega$ are finite, $\mathcal G_\omega$ says that there exist two starting positions for the cop and the robber such that, initially, one cannot name a finite time in which the game will end, but as soon as the robber makes his move, the cop can follow with a move such that now one can name a finite time for these two new positions.

The graph constructed in \cite{FI}, which also has maximum capture time $\omega$ but $\eta(u,v)\neq\omega$, it is $\mathcal G_\omega$ without the complete graph formed by the diagonal vertices.

This is in fact a general phenomenon. In other words, for the graph $\mathcal G_\gamma$ (where $\gamma$ is an infinite limit ordinal) there exist vertices $u$ and $v$ such that $\eta(u,v)=\gamma$, but if one removes the diagonal edges, the new graph is still cop-win, and it still has maximum capture time $\gamma$, but $\eta(u,v)\neq\gamma$ for any two vertices $u$ and $v$. The interested reader can verify that even if we remove the diagonal edges, Lemma~\ref{lemmaxaxis}, Lemma~\ref{lemmayaxis} (extended to the diagonal vertices) and Lemma~\ref{etalowerboundlemma} still hold, and hence the maximum capture time stays $\gamma$, which cannot be attained by $\eta(u,v)$ for any $u$ and $v$.

Therefore, every infinite limit ordinal is achieved as a CR-ordinal in the case where we either insist that $\eta(u,v)\neq\gamma$ for all $u$ and $v$, and in the case where we insists that there exist two vertices $u$ and $v$ such that $\eta(u,v)=\gamma$.

\vspace{1.5em}
\noindent{\textbf{Acknowledgment.} The second author would like to thank G-Research for generously funding his stay in Cambridge while undertaking this project.}

\Addresses

\begin{thebibliography}{99}
\bibitem{BN} A. Bonato and R. Nowakowski, The Game of Cops and Robbers on Graphs. \textit{American Mathematical Society} (2011), ISBN-13 978-0821853474.
\bibitem{BGH} A. Bonato, P. Gordinowicz and G. Hahn, Cops and Robbers ordinals of cop-win trees. \textit{Discrete Mathematics}, \textbf{340} (2017), 951 -- 956.
\bibitem{FI} T. Fl\' idr and M.-R. Ivan, A Cop-Win Graph wth Maximum Capture Time $\omega$. arXiv:2507.22276, (2025).
\bibitem{ILW} M.-R. Ivan, I. Leader and M. Walters, Constructible Graphs and Pursuit. \textit{Theoretical Computer Science}, \textbf{930} (2022), 196 -- 208.
\bibitem{NW} R. Nowakowski and P. Winkler, Vertex-to-vertex pursuit in a graph. \textit{Discrete Mathematics}, \textbf{43} (1983), 235--239.
\end{thebibliography}
\end{document}